\documentclass[11pt,english]{article}

\usepackage[margin = 2.5cm]{geometry}

\usepackage{amsthm}
\usepackage{amsmath}
\usepackage{amssymb}
\usepackage{setspace}
\usepackage{mathtools}
\usepackage{graphicx}
\usepackage[hidelinks]{hyperref}
\usepackage{thm-restate}
\usepackage{cleveref}
\usepackage{enumitem}
\setlist[itemize]{leftmargin=*} 
\usepackage{framed}
\usepackage{subcaption}

\usepackage{floatrow}

\theoremstyle{plain}

\newtheorem*{thm*}{Theorem}
\newtheorem{thm}{Theorem}
\Crefname{thm}{Theorem}{Theorems}

\newtheorem*{lem*}{Lemma}
\newtheorem{lem}[thm]{Lemma}
\Crefname{lem}{Lemma}{Lemmas}

\newtheorem*{claim*}{Claim}

\crefname{claim}{Claim}{Claims}
\Crefname{claim}{Claim}{Claims}

\Crefname{prop}{Proposition}{Propositions}

\newtheorem{cor}[thm]{Corollary}
\crefname{cor}{Corollary}{Corollaries}

\crefname{conj}{Conjecture}{Conjectures}

\Crefname{qn}{Question}{Questions}

\Crefname{obs}{Observation}{Observations}

\Crefname{ex}{Example}{Examples}

\theoremstyle{definition}

\Crefname{prob}{Problem}{Problems}

\Crefname{defn}{Definition}{Definitions}

\theoremstyle{remark}

\renewenvironment{proof}[1][]{\begin{trivlist}
\item[\hspace{\labelsep}{\bf\noindent Proof#1.\/}] }{\qed\end{trivlist}}

\newcommand{\remove}[1]{}

\newcommand{\RN}{r} 

\newcommand{\RE}{\overrightarrow{r}} 
\newcommand{\eps}{\varepsilon}

\newcommand{\dirpath}[1]{\overrightarrow{P_{#1}}}

\begin{document}


\title{The oriented size Ramsey number of directed paths}

\author{
    Shoham Letzter\thanks{
        ETH Institute for Theoretical Studies,
        ETH,
        8092 Zurich;
        e-mail: \texttt{shoham.letzter}@\texttt{eth-its.ethz.ch}.
    }
    \and
	Benny Sudakov\thanks{
	    Department of Mathematics, 
	    ETH, 
	    8092 Zurich;
	    e-mail: \texttt{benjamin.sudakov}@\texttt{math.ethz.ch}.     
		Research supported in part by SNSF grant 200021-175573.
	}
}

\date{}
\maketitle

\begin{abstract}

	\setlength{\parskip}{\medskipamount}
    \setlength{\parindent}{0pt}
    \noindent

	An \emph{oriented graph} is a directed graph with no bi-directed
	edges, i.e.\ if $xy$ is an edge then $yx$ is not an edge.
	The \emph{oriented size Ramsey number} of an oriented graph $H$, denoted by
	$\RE(H)$, is the minimum
	$m$ for which there exists an oriented graph $G$ with $m$ edges, such
	that every $2$-colouring of $G$ contains a monochromatic copy of $H$.

	In this paper we prove that the oriented size Ramsey number of the
	directed paths on $n$ vertices satisfies $\RE(\dirpath{n})
	= \Omega(n^2 \log n)$.  This improves a lower bound by Ben-Eliezer,
	Krivelevich and Sudakov. It also matches an upper bound by Buci\'c and
	the authors, thus establishing an asymptotically tight bound on
	$\RE(\dirpath{n})$.

	We also discuss how our methods can be used to improve the best known
	lower bound of the $k$-colour version of $\RE(\dirpath{n})$.

\end{abstract}
\section{Introduction}
	Given graphs $G$ and $H$, we write $G \rightarrow H$ if there is a
	monochromatic copy of $H$ in every $2$-edge-colouring of $G$.  The
	\emph{size Ramsey number} of a graph $H$, denoted by $\RN(H)$, is the
	minimum number of edges in $G$ over graphs $G$ satisfying $G \rightarrow H$.
	The concept of size Ramsey numbers was introduced by Erd\H{o}s, Faudree,
	Rousseau and Schelp \cite{erdos-faudree-rousseau-schelp} in 1972, and has
	received considerable attention since. A notable example is the size
	Ramsey number of a path on $n$ vertices, which was shown by Beck
	\cite{beck} to be linear in $n$, thus disproving a conjecture of
	Erd\H{o}s \cite{erdos}.
	 
	Here we consider an analogue of size Ramsey numbers for oriented graphs
	(recall that an oriented graph is a directed graph where at most one of
	$xy$ and $yx$ is an edge for every two vertices $x$ and $y$).  The
	\emph{oriented size Ramsey number} of an oriented graph $H$, denoted by
	$\RE(H)$, is the minimum number of edges of $G$ over \emph{oriented}
	graphs $G$ satisfying $G \rightarrow H$.

	In this note, we focus on the oriented size Ramsey number of the  
	directed path on $n$ vertices, denoted by $\dirpath{n}$. 
	Unlike the undirected case, it turns out that $\RE(\dirpath{n})$
	is not linear in $n$, as shown by Ben-Eliezer, Krivelevich and Sudakov
	\cite{ben-eliezer-krivelevich-sudakov}, who established the following bounds
	(where $c_1$ and $c_2$ are positive absolute constants). 
	$$ \frac{c_1 n^2 \log n}{(\log \log n)^3} \le \RE(\dirpath{n}) \le c_2n^2
	(\log n)^2.$$

	Recently, Buci\'c and the authors 
	\cite{bucic-letzter-sudakov-random-tournaments} improved the upper bound
	to $\RE(\dirpath{n}) \le c_3 n^2\log n$, by establishing a lower bound on
	the longest monochromatic path in $2$-coloured random tournaments,
	thereby bringing the upper and lower bounds very close together. The main aim of
	this note is to obtain a matching lower bound on $\RE(\dirpath{n})$, thus
	showing that $$\RE(\dirpath{n}) = \Theta(n^2\log n).$$ We achieve our aim
	in the following theorem. 

	\begin{thm} \label{thm:main}
		Let $G$ be a directed graph with at most $n^2 \log n$ edges. Then $G$
		can be $2$-coloured such that all monochromatic directed paths have
		length at most $169 n$.
	\end{thm}

	We prove \Cref{thm:main} in the next section and conclude the paper in
	\Cref{sec:conclusion} with a discussion of a generalisation to more
	colours.
	Throughout the paper we omit floor and ceiling signs whenever they are
	not crucial, and all logarithms are in base $2$.

\section{The proof}

	In our proof of \Cref{thm:main}, we follow the footsteps of Ben-Eliezer,
	Krivelevich and Sudakov \cite{ben-eliezer-krivelevich-sudakov}. The main
	difference is an improvement on their main tool in their proof of the lower
	bound, presented in \Cref{lem:special-sets} below. 

	Before stating
	the lemma we make a definition.
	We call a set $U$ of vertices in a directed graph \emph{$k$-special} if it is
	acyclic and its components (in the underlying graph) have order at most
	$k$. 
	
	In \cite{ben-eliezer-krivelevich-sudakov}, the authors proved that an
	oriented graph on $n$ vertices with at most $\eps n^2$ edges contains an
	acyclic subset of size at least $\frac{c\log n}{\eps \log(1/\eps)}$. It
	turns out that the proof of this statement, which is a directed version
	of a lemma by Erd\H{o}s and Szemer\'edi \cite{erdos-szemeredi}, can be
	adapted to give the following stronger statement which ensures the
	existence of a large special (and, in particular, acyclic) subset.

	\begin{lem} \label{lem:special-sets}
		Let $G$ be an oriented graph with $n$ vertices and at most $\eps n^2$
		edges,
		where $\eps > 1/n$. Then there is a $(\log n)$-special set of size at least
		$\frac{\log n}{20 \eps \log(1/\eps)}$.
	\end{lem}

	In the proof of \Cref{lem:special-sets} we shall need the following simple
	lemma. Its proof follows by induction from the fact that every oriented
	$m$-vertex graph has
	a vertex with out-degree at most $m/2$.

	\begin{lem} \label{lem:trivial-bound}
		Every oriented graph on $m$ vertices contains an acyclic subset of size at
		least $\log m$.
	\end{lem}

	\begin{proof}[ of \Cref{lem:special-sets}]
		We may assume that $\eps < 1/8$, otherwise the proof follows from
		\Cref{lem:trivial-bound}. Remove all vertices whose degree (in the
		underlying graph) is at least $4\eps n$, and denote the resulting
		graph by $G'$; note that $G'$ has at least $n - \frac{2\eps n^2}{4
		\eps n} =  n/2$ vertices.

		Let $U$ be a maximum $(\log n)$-special set in $G'$ (from now on, we call
		such sets \emph{special}). We may assume that $|U| < \frac{\log n}{20\eps
		\log(1/\eps)}$, and, since $\eps > 1/n$, also $|U| < n/20$.

		The number of edges between $U$ and $V(G') \setminus U$ is at most $4\eps
		n|U|$. Hence, the number of vertices in $V(G') \setminus U$ which have at
		least $10\eps |U|$ neighbours in $U$, is at
		most $2n/5$.  Therefore, the set $W$, of vertices outside of $U$ that have
		fewer than $10\eps |U|$ neighbours in $U$, has size at least $n/2 - |U| -
		2n/5 \ge n/20$.

		For each vertex $w \in W$, let $S_w$ be a subset of $U$ of size exactly
		$10\eps |U|$ that contains all the neighbours of $w$ in $U$. 
		Note that the number of possible sets $S_w$ is at most the following
		(using $\binom{n}{k} \le \left(\frac{en}{k}\right)^k$).
		$$ \binom{|U|}{10\eps |U|} 
		\le \left( \frac{e}{10\eps} \right)^{10\eps
		|U|} 
		\le 2^{\frac{\log n}{2\log(1/\eps)}\log (e/10\eps)} 
		\le \sqrt{n}.$$

		Hence, there is a subset $W'$ of $W$ of size at least $|W| / \sqrt{n}
		\ge \sqrt{n}/20$, for which $S_w$ is the same for all $w \in W'$.  By
		\Cref{lem:trivial-bound}, there is an acyclic subset $W''$ of $W'$
		whose size is at least $\frac{1}{2}\log n$.  Write $U' = (U \setminus
		S) \cup W''$, where $S = S_w$ for some $w \in W'$.  There are no
		edges between $U \setminus S$ and $W''$, hence, since $U$ and $W''$
		are acyclic, so is $U''$.  Furthermore, the components in $U'$ are
		contained in either $U$ or in $W''$, thus they have order at most
		$\log n$. It follows that $U'$ is special.  Finally, we have the
		following (note that $\eps < 1/2$).  
		$$|U'| \ge |U| - 10\eps |U| + \frac{\log n}{2} \ge |U| -
		\frac{\log n}{2\log 1/\eps} +  \frac{\log n}{2} > |U|.$$ 
		This is a
		contradiction to the choice of $U$ as a maximum special set.  The
		lemma follows.
	\end{proof}

	Before turning to the proof of \Cref{thm:main}, we mention the following 
	lemma.
	\begin{lem} \label{lem:colour-tournaments}
		Every acyclic graph on $m$ vertices can be $2$-edge-coloured such that
		every monochromatic directed path has length at most $\sqrt{m}$.
	\end{lem}

	\begin{proof}
		Let $G$ be an acyclic graph on $m$ vertices. Let $v_1, \ldots, v_m$ be an
		ordering of the vertices of $G$ such there are no edges $v_i v_j$ with $i
		> j$. Define $U_k = \{v_{1 + (k-1)\sqrt{m}},\ldots,v_{k\sqrt{m}}\}$ for
		$k \in [\sqrt{m}]$ (we
		assume that $\sqrt{m}$ is integer for simplicity).
		Colour edges inside the $U_i$'s red and edges between the $U_i$'s blue.
		It is easy to see that every monochromatic directed path has length at most
		$\sqrt{m}$.
	\end{proof}

	The following corollary easily follows: colour each of the components of a
	$k$-special set using the colouring described in
	\Cref{lem:colour-tournaments}.

	\begin{cor} \label{cor:colour-special-sets}
		The edges of a $k$-special set can be $2$-coloured such that
		monochromatic paths have length at most $\sqrt{k}$.
	\end{cor}

	We are now ready for the proof of \Cref{thm:main}.

	\begin{proof}[ of \Cref{thm:main}]
		Our aim is to partition $G$ into not too many  
		special sets (i.e.\ $(\log n)$-special) and a small remainder.
		The main tool is the following claim.
		\begin{lem} \label{lem:step-one}
			Let $G$ be an oriented graph  with at most $n^2 \log n$ edges.
			Then the vertices of $G$ can be
			partitioned into at most $\frac{160 n}{\sqrt{\log n}}$ special
			sets, and a remainder of at most $8 n \sqrt{\log n}$ vertices.
		\end{lem}

		\begin{proof}
			Our plan is very simple: we remove, one by one, a special set of
			maximum size, until we remain with at most $8 n \sqrt{\log n}$ vertices.
			In order to show that the number of special sets removed in such a
			process is not too large, we divide the process into stages.

			Let $\alpha$ be such that the number of vertices of $G$ is
			$\alpha n \sqrt{\log n}$; note that we may assume $\alpha > 8$,
			as otherwise we are done trivially.  Write $\alpha_i = \alpha /
			2^i$, for $0 \le i \le I$, where $I$ is smallest for which
			$\alpha_I \le 8$ holds (so $\alpha_I \ge 4$).  The first stage is
			the first part of the process described above, where special sets
			of maximum size are removed one by one, run until the first time
			when the number of vertices drops below $\alpha_1 n \sqrt{\log
			n}$. Similarly, the $i$-th stage consists of the part of the
			process which starts right after the end of the $(i-1)$-th stage,
			and lasts until the number of vertices drops below $\alpha_i n
			\sqrt{\log n}$.

			Write $\eps_i = \frac{n^2 \log n}{(\alpha_i n\sqrt{\log n})^2} =
			1/\alpha_i^2$ .
			By \Cref{lem:special-sets}, the special sets removed in the $i$-th
			stage have size at least $\frac{\log n}{20\eps_i \log(1/ \eps_i)} =
			\frac{\alpha_i^2 \log n}{40\log \alpha_i}$.
			Since the number of vertices removed is at most $\alpha_{i-1} n
			\sqrt{\log n}$, the number of special sets removed
			in the $i$-th stage, where $i \le I$, is at most the following. 
			\begin{equation} \label{eqn:one-step}
				\frac{\alpha_{i-1} n \sqrt{\log n}}{\frac{\alpha_i^2 \log
				n}{40\log(\alpha_i)}} 
				= \frac{80\log
				\alpha_i}{\alpha_i} \cdot \frac{n}{\sqrt{\log n}}
			\end{equation}
			Note that $\alpha_i = \alpha_I \cdot 2^{I - i} \ge 4 \cdot 2^{I
			- i} = 2^{2 + I - i}$, as $\alpha_I \ge 4$. Since
			$\frac{\log x}{x}$ is decreasing for $x \ge
			e$, we have that $\frac{\log(\alpha_i)}{\alpha_i} \le
			\frac{\log(2^{2 + I - i})}{2^{2 + I - i}} = \frac{2 + I - i}{2^{2
			+ I - i}}$ for $0 \le i \le I$.
			Hence, 
			\begin{equation} \label{eqn:total}
				\sum_{0 \le i \le I} \frac{80 \log \alpha_i}{\alpha_i}
				\le \sum_{2 \le i \le I + 2} \frac{80i}{2^i}
				\le 40 \sum_{i \ge 0} (i+1)2^{-i} 
				= 40 \sum_{i \ge 0}\sum_{j \le i} 2^{-i}
				= 40 \Big( \sum_{i \ge 0} 2^{-i} \Big)^2 
				= 160.
			\end{equation}
			It follows from \eqref{eqn:one-step} and \eqref{eqn:total}
			that the total number of special sets removed in this
			process is at most $\frac{160 n}{\sqrt{\log n}}$. Also, by
			definition of $\alpha_I$, the number of vertices remaining in the
			graph is at most $8 n \sqrt{\log n}$, as required.
		\end{proof}

		By \Cref{lem:step-one}, there is a partition of the vertices of $G$ 
		into at most $\frac{160 n}{\sqrt{\log n}}$ special sets and a
		remainder $W$ of at most $8n \log n$ vertices.
		By iterating \Cref{lem:trivial-bound} we can partition $W$ into at most
		$\frac{8n}{\sqrt{\log n}}$ acyclic sets of size $\log n$ and a
		remainder of at most $n$ vertices.
		As these acyclic sets are special, we thus obtain a partition
		$\{U_1, \ldots, U_l, W'\}$ of the vertices of $G$ where $U_i$ is
		special for $i \in [l]$, $l \le \frac{168 n}{\sqrt{\log n}}$, and
		$|W'| \le n$.
		
		We
		colour the edges inside the $U_i$'s with red and blue in such a
		way that monochromatic paths inside the $U_i$'s have length at most
		$\sqrt{\log n}$; this is possible due to
		\Cref{cor:colour-special-sets}.  We then colour edges between $U_i$
		and $U_j$ red if $i < j$ and blue if $i > j$. Finally, we colour
		edges into $W'$ red, edges from $W'$ blue, and colour the
		edges inside $W'$ arbitrarily.  Any monochromatic path in this
		colouring contains at most $\sqrt{\log n}$ vertices from each $U_i$
		and at most $n$ vertices from $W'$, hence it
		has length at most $169 n$, as required.
	\end{proof}

\section{Conclusion} \label{sec:conclusion}
	We conclude this note with a discussion of a generalisation to a
	$k$-coloured setting.  We shall consider $\RE(\dirpath{n}, k+1)$, the
	$(k+1)$-colour oriented size Ramsey number of $\dirpath{n}$.  The
	following bounds follow from \cite{ben-eliezer-krivelevich-sudakov} and
	\cite{bucic-letzter-sudakov-random-tournaments}.
	$$\frac{c_1 n^{2k}(\log n)^{1/k}}{(\log \log n)^{(k+2)/k}} 
	\le \RE(\dirpath{n},k+1) 
	\le c_2n^{2k} \log n.$$
	Using our approach, it is possible to improve the lower bound and obtain the
	following.
	$$\RE(\dirpath{n}, k+1) \ge c_1n^{2k} (\log n)^{2/(k+1)}.$$
	We give only a sketch of the proof.  As in the proof of \Cref{thm:main},
	we may partition a graph with at most  $n^{2k} (\log n)^{2/(k+1)}$ edges
	into at most $\frac{c n^k (\log n)^{1/(k+1)}}{\log n}$ special sets and a
	remainder of at most $n$ vertices. Using an analogue of
	\Cref{lem:colour-tournaments}, we may colour the edges of special
	sets with $k + 1$ colours such that monochromatic paths have length at most $(\log
	n)^{1/(k+1)}$. In order to colour the edges between special sets, we use
	the fact (see, e.g.\ \cite{ben-eliezer-krivelevich-sudakov}) that a
	directed graph on $m$ vertices can be $(k+1)$-coloured (where, if both
	$xy$ and $yx$ are edges, we may use a separate colour for each of them) such that
	monochromatic paths have length at most $cm^{1/k}$ and all colour classes
	are acyclic. We thus obtain a $(k+1)$-colouring in which monochromatic
	paths have length at most $c_2n$.

	Finally, we note that the upper bound in
	\cite{bucic-letzter-sudakov-random-tournaments} was obtained by showing
	that random tournaments have, with high probability, monochromatic paths
	of the required length in every $2$-colouring of their edges. 
	It seems plausible that a similar statement holds for $(k+1)$-colourings of
	random tournaments, and perhaps a matching upper bounds on
	$\RE(\dirpath{n},k+1)$ can be proved using the methods from
	\cite{bucic-letzter-sudakov-random-tournaments}.

\subsection*{Acknowledgements}
	We would like to thank Matija Buci\'c for stimulating discussions.
	The first author would like to acknowledge the support of Dr.~Max
	R\"ossler, the Walter Haefner Foundation and the ETH Zurich Foundation.

\bibliography{pathbib}
\bibliographystyle{amsplain}
\end{document}